\documentclass[pdflatex]{amsart}

\usepackage[utf8]{inputenc}
\usepackage[english]{babel} 
\usepackage{amsthm}
\usepackage{amsmath} 
\usepackage{amsxtra} 
\usepackage{mathrsfs}
\usepackage{amssymb} 
\usepackage{graphicx} 
\usepackage{stmaryrd} 
\usepackage{paralist} 
\usepackage{comment} 
\usepackage{mathtools}
\usepackage{ascmac}
\usepackage{tikz}
\usepackage{latexsym}

\theoremstyle{definition} 

\newtheorem{theorem}{Theorem}[section]
\newtheorem{definition}[theorem]{Definition}
\newtheorem{lemma}[theorem]{Lemma}
\newtheorem{fact}[theorem]{Fact}
\newtheorem{prop}[theorem]{Proposition}

\newtheorem*{theorem*}{Theorem}

\newtheorem*{claim*}{Claim}
\newtheorem*{subclaim*}{Subclaim}

\newtheorem{question}[theorem]{Question}

\newtheorem*{lemma*}{Lemma} 
\newtheorem*{prop*}{Proposition}

\newcommand{\mbb}{\mathbb} 
\newcommand{\mbbA}{\mathbb{A}}
\newcommand{\mbbB}{\mathbb{B}}
\newcommand{\mbbC}{\mathbb{C}}
\newcommand{\mbbD}{\mathbb{D}}

\newcommand{\mbbL}{\mathbb{L}}

\newcommand{\mbbP}{\mathbb{P}}
\newcommand{\mbbQ}{\mathbb{Q}} 
\newcommand{\mbbR}{\mathbb{R}}

\newcommand{\ordD}{(D, \leq_D)}
\newcommand{\ordE}{(E, \leq_E)}
\newcommand{\ordP}{(P, \leq_P)}
\newcommand{\ordQ}{(Q, \leq_Q)}

\newcommand{\mcalM}{\mathcal{M}}
\newcommand{\mcalN}{\mathcal{N}}

\newcommand{\mcalP}{\mathcal{P}}

\newcommand{\Coll}{\mathrm{Coll}}
 
\newcommand{\collkappa}{\Coll(\omega,{<} \kappa)} 
 
\newcommand{\collalpha}{\Coll(\omega,\alpha)}
\newcommand{\restrict}{\upharpoonright} 
  
\newcommand{\CD}{\mathbb{C}\ast\dot{\mathbb{D}}}
\newcommand{\cof}{\mathrm{cof}}
\newcommand{\collapse}[1]{\Coll (\omega, {#1})} 

\newcommand{\ZFC}{\mathsf{ZFC}} 
\newcommand{\ZF}{\mathsf{ZF}} 
\newcommand{\CH}{\mathsf{CH}}
\newcommand{\AD}{\mathsf{AD}} 
\newcommand{\AC}{\mathsf{AC}} 
\newcommand{\DC}{\mathsf{DC}}

\newcommand{\dom}{\mathrm{dom}}
\newcommand{\stem}{\mathrm{stem}}

\newcommand{\suc}{\mathrm{Suc}}

\newcommand{\Ord}{\mathrm{Ord}}
\newcommand{\WO}{\mathrm{WO}}

\newcommand{\LR}{L(\mathbb{R})}
\newcommand{\VRVG}{V( \mbbR^{V[G]})}

\newcommand{\NWD}{\mathrm{NWD}}
\newcommand{\cNWD}{\mathrm{cNWD}}

\newcommand{\meagerideal}{\left(\mathcal{M}, \subseteq \right)}
\newcommand{\nullideal}{\left(\mathcal{N}, \subseteq \right)}
\newcommand{\comeagerfilter}{ ( \hat{\mathcal{M}}, \supseteq ) }

\newcommand{\dominating}{\left(\omega^\omega, \leq^* \right)}
\newcommand{\omegaone}{\left(\omega_1, \leq \right)}

\title[Generalized Tukey reducibility between $\sigma$-directed sets]{Generalized Tukey reducibility between $\sigma$-directed sets}
\author{Hiroshi Sakai}
\thanks{Graduate School of Mathematical Sciences, The University of Tokyo,  
3‑8‑1 Komaba, Meguro‑ku, Tokyo, 153‑8914, Japan. 
E-mail : hsakai@ms.u-tokyo.ac.jp}
\author{Toshimasa Tanno}
\thanks{Graduate School of System Informatics, Kobe University, 1-1 Rokko-dai, Nada-ku, Kobe, 657-8501, Japan. 
E-mail : 211x503x@stu.kobe-u.ac.jp.
This author was supported by  JST SPRING, Grant Number JPMJSP2148.}
\date{}

\begin{document}

\begin{abstract}
  We introduce the pre-Tukey reducibility, a generalization of the Tukey reducibility between directed sets that works well in $\ZF$.
  We investigate the pre-Tukey reducibility between several $\sigma$-directed sets under assumptions on sets of reals, which hold in the Solovay model and in $\LR$ satisfying $\AD$.
\end{abstract}

\maketitle

\section{Introduction}

In set theory, directed sets form an important class of objects. 
In order to compare directed sets, Tukey introduced a fundamental notion known as Tukey reducibility \cite{tukey1940convergence}.
He showed that if two directed sets are equivalent with respect to the Tukey reducibility then they are cofinally similar, that is, both can be embedded cofinally into a directed set. 
This reducibility and the generated equivalence classes (called cofinal types) have been extensively studied within the framework of $\ZFC$.
In particular, in the area of cardinal characteristics, the Tukey reducibility among various $\sigma$-directed orders on the reals has been thoroughly investigated as a means of comparing their cofinalities, the smallest cardinalities of cofinal subsets.
For example, the Tukey relations 
$(\omega^\omega, \leq^*)\preceq_T (\mcalM, \subseteq) \preceq_T (\mcalN, \subseteq)$ 
are provable in $\ZFC$, and it implies the inequality of cardinal characteristics $\mathfrak{d}\leq \cof (\mcalM) \leq \cof (\mcalN)$.
Here $\leq^*$ is the usual domination relation on $\omega^\omega$, that is, $x \leq^* y$ if and only if $x(n) \leq y(n)$ for all but finitely many $n < \omega$. $\mcalM$ and $\mcalN$ denote the meager ideal and the null ideal over $2^\omega$.

The definition of cofinalities and the basics of the Tukey reducibility heavily rely on $\AC$.
To compare the cofinal types of directed sets in a context without $\AC$, we introduce the notion of the pre-Tukey reducibility.
For directed sets $\ordD$ and $\ordE$, we say that $\ordD$ is pre-Tukey reducible to $\ordE$ and write $\ordD \preceq_{pT} \ordE$ if $\ordD$ is Tukey reducible to a completion of $\ordE$, the poset of non-empty $\leq_E$-upward closed subsets of $E$ ordered by reverse inclusion.
We also verify that the notion of the pre-Tukey reducibility works well without $\AC$.
That is, without using $\AC$ we show that if  directed sets are equivalent with respect to pre-Tukey reducibility then they are cofinally similar.

 We also investigate the pre-Tukey reducibility among several concrete directed sets.
 More precisely, we investigate $\sigma$-directed sets

\begin{center}
  $( \omega^\omega , \leq^* ) , \ ( \mcalM , \subseteq ) , \ ( \mcalN , \subseteq ) , \ ( \omega_1 , \leq ) , \ ( [ \omega^\omega ]^\omega , \subseteq )$
\end{center}
in models of $\ZF + \DC$ with the condition $(\star)$, which hold in the Solovay model and in $L(\mathbb{R})$ satisfying $\AD$.
$(\star)$ is the conjunction of the following two conditions:

\begin{center}
  \begin{enumerate}
    \item $\omega_1$ is inaccessible in $L[S]$ for any set $S$ of ordinals.
    \item Every subset of $\omega^\omega$ is $\infty$-Borel.
  \end{enumerate}
\end{center}

The pre-Tukey relations 
$(\omega^\omega, \leq^*) \preceq_{pT} (\mcalM, \subseteq) \preceq_{pT} (\mcalN, \subseteq)$,
which corresponds to the inequalities of cardinal characteristics 
$\mathfrak{d} \leq \cof (\mcalM) \leq \cof (\mcalN)$,
can be proved in $\ZF + \DC$.
We prove the pre-Tukey relations between above five $\sigma$-directed sets in models of $\ZF+\DC+(\star)$, which are illustrated in Figure \ref{fig:preTukey_directed}.
In the figure, $D \rightarrow E$ means that $D \preceq_{pT} E$, and $D \dashrightarrow E$ means that $D \not\preceq_{pT} E$ for directed sets $D$ and $E$.
In particular, in models of $(\star)$, the cofinal types of the above directed sets are different from each other.

  \begin{figure}[ht] \label{fig:preTukey_directed}
  \begin{tikzpicture}[node distance=2cm, auto]

  \node (A) at (0, 0) {$(\omega^\omega, \leq^*)$};
  \node (B) at (2, 0) {$(\mcalM, \subseteq)$};
  \node (C) at (4, 0) {$(\mcalN, \subseteq)$};
  \node (D) at (6, -1) {$( [ \omega^\omega ]^\omega, \subseteq )$};
  \node (E) at (2, -2) {$(\omega_1, \leq)$};
  
  \draw[->] (A) -- (B) node[midway, above] {};
  \draw[->] (B) -- (C) node[midway, above] {};
  \draw[->] (C) -- (D) node[midway, above] {};
  \draw[->] (E) -- (D) node[midway, above] {};
  
  \draw[dashed, <-] (A) to[bend right=30] (B);
  \draw[dashed, <-] (B) to[bend right=30] (C);
  \draw[dashed, <-] (C) to[bend right=30] (D);
  \draw[dashed, <-] (E) to[bend right=30] (D);
  
  \foreach \x in {A, B, C} {
      \draw[dashed, <->] (\x) -- (E);
  }
  \end{tikzpicture}
  \caption{pre-Tukey relations in $\ZF+\DC+(\star)$}
  \end{figure}
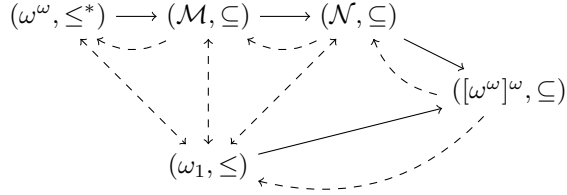

The base theory of Set Theory in this paper is $\ZF + \DC$. We work in $\ZF + \DC$, and all theorems, lemmas and facts are those of $\ZF + \DC$, unless otherwise stated.

\section{Preliminaries}
Here we present our notation and basic facts used in this paper.

First we give our notation on ordered sets.
We deal with several kinds of orderings.
A \textit{pre-ordering} is a binary relation which is reflexive and transitive.
A \textit{poset} is a pre-ordered set.
A \textit{partial ordering} is an antisymmetric pre-ordering.
A \textit{directed set} ($\sigma$-directed set) is a poset in which any finite (countable) elements have an upper bound.

Let $\ordP$ be a poset.
For $p,q \in P$ we write $p <_P q$ if $p \leq_P q$, and $q \not\leq_P p$.
For $p\in P$ let $p\uparrow_{\leq_P}=\{q\in P\mid p\leq_P q\}$ and $p\downarrow_{\leq_P} =\{q\in P\mid q\leq_P p\}$.

Let $\ordP$ and $\ordQ$ be posets and $f$ be a function from $P$ to $Q$.
$f$ is called an \textit{embedding} if 
$f(p_0) \leq_{Q} f(p_1)$ exactly when $p_0 \leq_P p_1$ for any $p_0, p_1 \in P$.
We say that $f$ is \textit{cofinal} if $f[P]$ is cofinal in $Q$.
$P$ is (\textit{cofinally}) \textit{embeddable} into $Q$ if there is a (cofinal) embedding from $P$ to $Q$.

\vskip.5\baselineskip

Next, we present our notation on forcings. In this paper, we only deal with forcings over models of $\ZFC$.

Let $\mbb{P}$ be a forcing notion and $\dot{x}$ be a $\mbb{P}$-name.
Then, for a $\mbb{P}$-generic filter $G$ over $V$, let $\dot{x}[G]$ denote the evaluation of $\dot{x}$ by $G$.

For an inaccessible cardinal $\kappa$, $\collkappa$ denotes the L\'{e}vy collapse which collapses $\kappa$ to $\omega_1$. That is, $\collkappa$ consists of all finite partial functions $p$ from $\kappa \times \omega$ to $\kappa$ such that $p( \alpha , n ) < \alpha$ for all $( \alpha , n ) \in \dom (p)$. Its order is reverse inclusion. Recall that $\collkappa$ is homogeneous.

Let $\mbbC$ denote the Cohen forcing notion, i.e.~$2^{< \omega}$ ordered by reverse inclusion.
Let $\mbbB$ be the random forcing notion. That is, $\mbbB$ consists of all Borel subsets of $2^\omega$ which has a positive Lebesgue measure, and $\mbbB$ is ordered by inclusion.

Let $W$ be a transitive inner model of $\ZFC$. A \textit{Cohen real} over $W$ is $x \in 2^\omega$ such that $x = \bigcup G$ for some $\mbbC$-generic filter over $W$. A \textit{random real} over $W$ is $x \in 2^\omega$ such that $\{ x \} = \bigcap H$ for some $\mbbB$-generic filter over $W$.
Recall that $x \in 2^\omega$ is a Cohen real over $W$ if and only if $x$ belongs to all Borel comeager sets coded in $W$.
Also, $x \in 2^\omega$ is a random real over $W$ if and only if $x$ belongs to all Borel conull sets coded in $W$.
We let $C^W$ and $R^W$ be the collections of all Cohen reals and random reals over $W$, respectively.

\section{Pre-Tukey relation}

In this section, we introduce the pre-Tukey relation, which weakens the Tukey relation and works well without $\AC$.
First we recall the original Tukey relation.

\begin{definition}
  Let $(D, \leq_D)$ and $(E, \leq_E)$ be directed sets.
  A function $f \colon D\rightarrow E$ is a \textit{Tukey map} if $f$ maps unbounded sets in $(D, \leq_D)$ into unbounded sets in $(E, \leq_E)$.
  A function $f \colon D\rightarrow E$ is a \textit{convergent map} if $f$ maps cofinal sets in $(D, \leq_D)$ into cofinal sets in $(E, \leq_E)$.

  We say that $(D,\leq_D)$ is \textit{Tukey reducible} to $(E, \leq_E)$ if there is a Tukey map from $(D,\leq_D)$ to $(E, \leq_E)$, and we write $(D, \leq_D) \preceq_T (E,\leq_E)$.
  Let $(D, \leq_D) \equiv_T (E, \leq_E)$ if $(D, \leq_D) \preceq_T (E, \leq_E)$ and $(E, \leq_E) \preceq_T (D, \leq_D)$.
\end{definition}

$\equiv_T$ is an equivalence relation between directed sets.
An equivalence class of $\equiv_T$ is called a \emph{cofinal type}.

$f \colon D\rightarrow E$ is Tukey if and only if
\[
  \forall e\in E\  \exists d\in D\  \forall d'\in D \ 
    (f(d')\leq_E e \Rightarrow d'\leq_D d) \, ,
\]
and $f$ is convergent if and only if 
\[
  \forall e\in E\  \exists d\in D\  \forall d'\in D \ 
    (d\leq_D d' \Rightarrow e\leq_E f(d')) \, .
\]

The following are basic facts on the Tukey reducibility.

\begin{fact}[\cite{tukey1940convergence}] \label{converse}
  Assume $\AC$.
  Let $(D, \leq_D)$ and $(E, \leq_E)$ be directed sets.
  $\left(D, \leq_D \right) \preceq_T (E, \leq_E)$ if and only if there is a convergent map from $(D, \leq_D)$ to $(E, \leq_E)$.
\end{fact}

\begin{fact}[\cite{tukey1940convergence}] \label{cof_embedding}
  Assume $\AC$.
  Let $(D, \leq_D)$ and $(E, \leq_E)$ be directed sets.
  Then $(D, \leq_D) \equiv_T (E, \leq_E)$ if and only if there is a directed set $(P, \leq_P)$ into which both $(D, \leq_D)$ and $(E, \leq_E)$ are cofinally embeddable.
\end{fact}

We introduce a generalization of Tukey relation, which we call pre-Tukey relation.
\begin{definition}\label{completion}
  Let $(D, \leq_D)$ be a directed set.
  $(D, \leq_D)^*$ is the directed set of all non-empty $\leq_D$-upward closed subsets of $D$ ordered by reverse inclusion. 
\end{definition}

We note that $(D, \leq_D)^*$ is a completion of $(D, \leq_D)$. 
In fact, $(D, \leq_D)^*$ is a  directed set in which every bounded subset has a least upper bound, and 
a map $d \mapsto \{d'\in D\mid d'\geq_D d\}$ is a cofinal embedding from $(D, \leq_D)$ to $(D, \leq_D)^*$.

\begin{definition}\label{Tukey_completion}
  Let $(D, \leq_D)$ and $(E, \leq_E)$ be directed sets.
  We define $(D, \leq_D)\preceq_{pT} (E, \leq_E)$ if $(D, \leq_D) \preceq_T (E, \leq_E)^*$.
  We call this relation $\preceq_{pT}$ \textit{pre-Tukey relation}.
  We write $\ordD \equiv_{pT} \ordE$ if $\ordD \preceq_{pT} \ordE$ and $\ordE \preceq_{pT} \ordD$.
\end{definition}

It is easy to see that $\preceq_{pT}$ is a pre-ordering, and $\equiv_{pT}$ is an equivalence relation. In the rest of this section, we study basics on the pre-Tukey relation.

The pre-Tukey relation is characterized by the existence of maps that generalize Tukey maps and convergent maps.

\begin{definition}
  Let $\ordD$ and $\ordE$ be directed sets.
  \begin{enumerate}
      \item A \textit{pre-Tukey map} from $\ordD$ to $\ordE$ is a map $\pi \colon D \rightarrow \mcalP(E)$ such that $\pi (d) \neq \emptyset$ for any $d \in D$, and
      \[
      \forall e\in E\  \exists d\in D\  \forall d'\in D \ 
      (\pi(d')\cap e\downarrow_E \neq \emptyset \Rightarrow d' \leq_D d) \, .
      \]
      \item A \textit{pre-convergent map} from $\ordD$ to $\ordE$ is a map $\sigma \colon D \rightarrow \mcalP(E)$ such that $\sigma (d) \neq \emptyset$ for any $d \in D$, and
      \[
      \forall e \in E\  \exists d\in D\  \forall d'\in D\ 
      (d \leq_{D} d' \Rightarrow \sigma(d')\subseteq e \uparrow_{E}) \, .
      \]
  \end{enumerate}
\end{definition}

\begin{lemma}\label{Tukey_map}
  The following are equivalent for directed sets $\ordD$ and $\ordE$.
  \begin{enumerate}
    \item $\ordD \preceq_{pT} \ordE$.
    \item There is a pre-Tukey map from $\ordD$ to $\ordE$.
    \item There is a pre-convergent map from $\ordE$ to $\ordD$.
  \end{enumerate}
\end{lemma}

\begin{proof}
  $(1)\Rightarrow (2) \colon $
  Suppose $\pi \colon D \rightarrow E^*$ is a Tukey map, where $E^*$ is the set of all non-empty $\leq_E$-upwards closed sets.
  For each $e \in E$, if $d \in D$ witnesses that $\pi \colon D \rightarrow E^*$ is Tukey for $e \uparrow_{\leq_E} \in E^*$, then $d$ witnesses that $\pi \colon D \rightarrow \mcalP(E)$ is pre-Tukey for $e$.

  $(2)\Rightarrow (1) \colon $
  Suppose $\pi \colon D\rightarrow \mcalP(E)$ is a pre-Tukey map.
  Define $f \colon D\rightarrow E^*$ by letting $f(d)$ be the $\leq_E$-upward closure of $\pi(d)$.
  Then it is easy to check that $f$ is a Tukey map from $\ordD$ to $\ordE^*$.

  $(2)\Rightarrow (3) \colon $
  Suppose $\pi \colon D\rightarrow \mcalP(E)$ is a pre-Tukey map.
  Define $\sigma \colon E\rightarrow \mcalP(D)$ by 
  \[
    \sigma(e)=\{d\in D\mid \forall d'\in D\  
    (\pi(d')\cap e\downarrow_{\leq_E} \neq \emptyset \Rightarrow d'\leq_D d)\} \, .
  \]
  Since $\pi$ is a pre-Tukey map, $\sigma(e)$ is non-empty for any $e\in E$.
  It is easy to check that each $e \in \pi(d)$ witnesses that $\sigma$ is a pre-convergent map for each $d \in D$.
  
  $(3)\Rightarrow (2) \colon $
  Suppose $\sigma \colon E\rightarrow \mcalP(D)$ is a pre-convergent map.
  Define $\pi \colon D\rightarrow \mcalP(E)$ by
  \[
    \pi(d)=\{e\in E\mid \forall e'\in E\ 
    (e\leq_{E} e' \Rightarrow \sigma(e')\subseteq d\uparrow_{\leq_D})\} \ .
  \]
 Since $\sigma$ is a pre-convergent map, $\pi(d)$ is non-empty for any $d\in D$.
 It is easy to check that each $d \in \sigma (e)$ witnesses that $\pi$ is a pre-Tukey map for each $e \in E$.
\end{proof}

For a pre-Tukey map $\pi$ from $\ordD$ to $\ordE$, let $\pi' \colon D \rightarrow \mcalP (E)$ be a function such that $\pi ' (d)$ is the $\leq_E$-upward closure of $\pi (d)$.
Then $\pi'$ is also a pre-Tukey map.
The same holds for pre-convergent maps.

Next we observe the relationships between Tukey (convergent) maps and pre-Tukey (pre-convergent) maps.

\begin{lemma}\label{Tukey_imply_preTukey}
  Let $\ordD$ and $\ordE$ be directed sets.
  For $f \colon D\rightarrow E$ define $\pi_f \colon D\rightarrow \mcalP(E)$ by $\pi_f (d)=f(d) \uparrow_{\leq_E}$.
  \begin{enumerate}
    \item If $f$ is a Tukey map, then $\pi_f$ is a pre-Tukey map.
    \item If $f$ is a convergent map, then $\pi_f$ is a pre-convergent map.
  \end{enumerate}
\end{lemma}
\begin{proof}
  Suppose $f$ is a Tukey map.
  To show that $\pi_f$ is pre-Tukey, let $e\in E$.
  Then there is $d\in D$ such that if $f(d')\leq_E e$ then $d'\leq_D d$.
  We show that $d$ witnesses that $\pi_f$ is pre-Tukey for $e$.
  Suppose $\pi_f (d') \cap e \downarrow_{\leq_E} \neq\emptyset$.
  Then $e \in \pi_f (d')$, and so $f(d') \leq_E e$.
  Thus $d' \leq_D d$.

  Suppose $f$ is a convergent map.
  To show that $\pi_f$ is pre-convergent, let $e\in E$.
  Then there is $d\in D$ such that $e \leq_E f(d')$ for any $d' \geq_D d$.
  Then $\pi_f(d')\subseteq e\uparrow_{E}$ for any $d' \geq_D d$.
\end{proof}

\begin{lemma}\label{preTukey_imply_Tukey}
  Let $\ordD$ and $\ordE$ be directed sets and $\pi \colon D \rightarrow \mcalP(E)$.
  Suppose that $f \colon D\rightarrow E$, and $f(d)\in \pi(d)$ for each $d\in D$.
  \begin{enumerate}
   \item If $\pi$ is a pre-Tukey map, then $f$ is a Tukey map.
   \item If $\pi$ is a pre-convergent map, then $f$ is a convergent map.
  \end{enumerate}
\end{lemma}
\begin{proof}
 Suppose $\pi$ is a pre-Tukey map and let $e\in E$.
 We can take $d\in D$ such that if $\pi(d')\cap e\downarrow_{\leq_E}\neq \emptyset$ then $d'\leq_D d$.
 If $f(d') \leq_E e$, then $f(d')\in \pi (d') \cap e\downarrow_{\leq_E}$, and so $d'\leq_D d$.
 Thus $f$ is a Tukey map.

 Suppose that $\pi$ is a pre-convergent map and let $e\in E$. 
 We can take $d\in D$ such that if $d\leq_D d'$ then $\pi(d')\subseteq e\uparrow_{\leq_E}$.
 Since $f(d') \in \pi (d')$, if $d\leq_D d'$ then $e\leq_E f(d')$.
 Thus $f$ is a convergent map.
\end{proof}

Next, we prove the analog of Fact \ref{cof_embedding}.
  
\begin{theorem}
  Let $\ordD$ and $\ordE$ be directed sets.
  $\ordD \equiv_{pT} \ordE$ if and only if there is a directed set $\ordP$ into which both $\ordD$ and $\ordE$ are cofinally embeddable.
\end{theorem}

\begin{proof}
  Assume that there is a directed set $\ordP$ into which both $\ordD$ and $\ordE$ is cofinally embeddable.
  We may assume that $D$ and $E$ are subsets of $P$.
  By symmetry, it suffices to construct a pre-Tukey map from $D$ to $E$.
  Define $\pi \colon D\rightarrow \mcalP(E)$ by
  $\pi(d)=\{e\in E\mid d\leq_P e\}$.
  We show that $\pi$ is pre-Tukey.
  Suppose $e\in E$, then there is $d\in D$ with $e\leq_P d$.
  If $\pi(d')\cap e\downarrow_{\leq_E} \neq\emptyset$
  then $d'\leq_P e$, and thus $d'\leq_D d$.

  Conversely, assume that $\ordD \equiv_{pT} \ordE$. We may assume $D \cap E = \emptyset$.
  Let $\pi_0 \colon D\rightarrow \mcalP(E)$ and 
  $\pi_1 \colon E\rightarrow \mcalP(D)$ be pre-Tukey maps.
  We may assume that each value of $\pi_0$ and $\pi_1$ is upward closed.
  Define $\sigma_0 \colon E\rightarrow \mcalP(D)$ and $\sigma_1 \colon D\rightarrow \mcalP(E)$ by
  \begin{eqnarray*}
    \sigma_0(e)=\{d\in D\mid \forall d'\in D \ 
    (\pi_0 (d')\cap e\downarrow_E \neq\emptyset \Rightarrow d'\leq_D d)\} \, , \\
    \sigma_1(d)=\{e\in E\mid \forall e'\in E \ 
    (\pi_1 (e')\cap d\downarrow_D \neq\emptyset \Rightarrow e'\leq_E e)\} \, .
  \end{eqnarray*}
  Note that $\sigma_0$ and $\sigma_1$ are pre-convergent maps and that $\sigma_0 (e)$ and $\sigma_1 (d)$ are upward closed in $D$ and $E$, respectively.

  Define $\pi \colon D\rightarrow \mcalP(E)$ by $\pi(d)=\pi_0 (d)\cap \sigma_1(d)$ and $\sigma \colon E\rightarrow \mcalP(D)$ by $\sigma(e)=\pi_1 (e)\cap \sigma_0(e)$.
  Since $\ordE$ is directed, and $\pi_0 (d) , \sigma_1 (d)$ are $\leq_E$-upward closed, we have $\pi (d) \neq \emptyset$. Similarly, $\sigma (e) \neq \emptyset$.

  By the construction of $\sigma_0$ and $\sigma_1$, and the fact that $\sigma(e)\subseteq \sigma_0 (e)$ and $\pi(d)\subseteq \sigma_1(d)$, we have the following for all $d\in D$ and $e\in E:$
  \begin{enumerate}
    \item $\pi(d)\cap e\downarrow _{E}\neq \emptyset
    \Rightarrow \sigma(e)\subseteq d\uparrow_{D}\ ;$
    \item $\sigma(e)\cap d\downarrow _{D}\neq \emptyset
    \Rightarrow \pi(d)\subseteq e\uparrow_{E}$\ .
  \end{enumerate}
  Let $P = D \cup E$, and define an ordering $\leq_P$ on $P$ as follows.
  \begin{itemize}
    \item $\leq_P \cap \ (D\times D)= \ \leq_D\ ;$
    \item $\leq_P \cap \ (E\times E)= \ \leq_E\ ;$
    \item For $d\in D$ and $e\in E$,
    $d\leq_P e \Leftrightarrow \exists d'\geq_D d\ (e\in \pi(d'))\ ;$
    \item  For $d\in D$ and $e\in E$, 
    $e\leq_P d \Leftrightarrow \exists e'\geq_E e\ (d\in \sigma(e'))$\ .
  \end{itemize}
  It suffices to prove that $\leq_P$ is transitive, and both $D$ and $E$ are cofinal in $\ordP$.

  For the transitivity of $\leq_P$, we only show that if $d_0\leq_P e\leq_P d_1$ then $d_0\leq_P d_1$ for all $d_0, d_1 \in D$ and $e\in E$.
  The proofs of the other cases are easy or similar to this case.
  We take $d'\geq_D d_0$ such that $e\in \pi(d')$ and $e'\geq_E e$ such that $d_1\in \sigma(e')$.
  Then $e\in \pi(d')\cap e'\downarrow_{\leq_E}$, so $\sigma(e')\subseteq d'\uparrow_{\leq_D}$ by (1).
  Thus $d_0 \leq_D d'\leq_D d_1$, and so $d_0 \leq_P d_1$.

  We show that $D$ and $E$ are cofinal in $\ordP$.
  To prove that $D$ is cofinal, suppose $e \in E$.
  Since $\sigma (e) \neq \emptyset$, we can take $d \in \sigma (e)$.
  Then $e \leq_P d \in D$ by the definition of $\leq_P$.
  Similarly $E$ is cofinal in $\ordP$.
\end{proof}

We note that the pre-Tukey relation and the Tukey relation need not coincide without $\AC$.

Let $\mathrm{WO}$ be the set of all $x \in {}^\omega \omega$, which codes a well-ordering of $\omega$ in the usual way.
For each $x \in \mathrm{WO}$ let $o(x)$ be the order-type of the well-ordering coded by $x$.
For $x,y \in \mathrm{WO}$ we write $x \trianglelefteq y$ if either $o(x) < o(y)$, or $o(x) = o(y)$ and $x \leq_\mathrm{lex} y$, where $\leq_\mathrm{lex}$ is the lexicographic ordering on $\omega^\omega$.
Then $\trianglelefteq$ is a linear ordering on $\mathrm{WO}$.

Assume that there is no injection from $\omega_1$ to $\omega^\omega$.
If there is a Tukey map from $(\omega_1, \leq)$ to $(\WO, \trianglelefteq)$ then we can construct an injection from $\omega_1$ to $\omega^\omega$. Thus $(\omega_1, \leq)\npreceq_T (\mathrm{WO}, \trianglelefteq).$
On the other hand, $\pi \colon \omega_1 \to \mathcal{P}( \mathrm{WO} )$ defined by
\[
\pi ( \xi ) = \{ x \in \mathrm{WO} \mid o(x) = \xi \}
\]
is a pre-Tukey and pre-convergent map. So 
$( \omega_1 , \leq ) \equiv_{pT} ( \mathrm{WO} , \trianglelefteq )$.

\section{Pre-Tukey relation among $\sigma$-directed sets}
In this section, we investigate the pre-Tukey relations between 
\begin{center}
$\left( \omega^\omega, \leq^* \right), 
\left( \mcalM, \subseteq \right), 
\left(\mcalN, \subseteq \right), 
\left(\omega_1, \leq\right)$, 
$([\omega^\omega]^\omega, \subseteq)$.
\end{center}

We consider these directed sets under an assumption on sets of reals, which holds in the Solovay model and in $L(\mathbb{R})$ satisfying $\AD$.

\begin{definition}[Woodin]
  $X \subseteq \omega^\omega$ is $\infty$-\textit{Borel} if there exist a set $S$ of ordinals and a formula $\varphi$ such that 
  \begin{center}
    $x\in X \Leftrightarrow
    L[S][x]\models \varphi(S,x)$
  \end{center}
  for all $x\in \omega^\omega$.
  We call this pair $(S, \varphi)$ a \textit{Solovay code for} $X$.
\end{definition}

\begin{definition}
  Let $(\star)$ denote the following statements:
  \begin{enumerate}
    \item $\omega_1$ is inaccessible in $L[S]$ for any set $S$ of ordinals.
    \item Every subset of $\omega^\omega$ is $\infty$-Borel. 
  \end{enumerate}
\end{definition}

For an inaccessible cardinal $\kappa$ and a $\collkappa$-generic filter $G$ over $V$, 
we let
$\VRVG = \mathrm{HOD}^{V[G]}_{V\cup \mbbR^{V[G]}}$, the collection of hereditarily ordinal defiable sets in $V[G]$ with parameters from $V\cup \mbbR^{V[G]}$.
We call $\VRVG$ the Solovay model for $\kappa$.
\begin{fact}\label{fact_Solovay_LR}
  \begin{enumerate}
    \item Assume $\AC$ in $V$. Let $\kappa$ be an inaccessible cardinal and let $G$ be a $\collkappa$-generic filter over $V$.
    Then $(\star)$ holds in $\VRVG$.
    \item If $\LR\models \AD$ then $\LR\models (\star)$.
  \end{enumerate}
\end{fact}

\begin{proof}
  $(1)$ : We work in $V[G]$.
  Suppose that $\lambda$ is an ordinal and $S \subseteq \lambda$.
  By the homogeneity of $\collkappa$, there are $v\in V$, $r \in \omega^\omega$ and a formula $\varphi$ such that 
  \begin{center}
    $\alpha \in S \Leftrightarrow V[r] \models 1_{\collkappa} \Vdash_{\collkappa} \varphi (v, r, \alpha)$
  \end{center}
  for all $ \alpha < \lambda$.
  Then $S \in V[r]$, so $L[S] \subseteq V[r]$.
  Since $\kappa=\omega_1$ is inaccessible in $V[r]$,
  $\omega_1$ is inaccessible in $L[S]$.

  We suppose that $X \subseteq \omega^\omega$ and show that $X$ is $\infty$-Borel.
  By the reflection principle, there are $v \in V$, $r \in \omega^\omega$, a formula $\psi$ and an ordinal $\gamma$
  such that
\[
x \in X \ \Leftrightarrow \ 
V_{\gamma}^{V[G]} \models \psi (r,v,x) \, .
\]
for all $x\in\omega^\omega$.
Since $\AC$ holds in $V$, we can take a set $S'$ of ordinals such that ${V_\gamma}^{V[G]} = {V_\gamma}^{L[S'][G]}$. Then

\begin{align*}
  x \in X
  &\Leftrightarrow L[S'][G]\models `` V_\gamma \models \psi (r, v, x) ''\\
  & \Leftrightarrow L[S'][r][x]\models 1_{\collkappa} \Vdash_{\collkappa} `` V_\gamma \models \psi (r, v, x) ''\ .
\end{align*}

Let $S$ be a set of ordinals coding $S', r, v, \kappa$ and $\gamma$ such that $L[S] = L[S'][r]$.
Let $\varphi(S, x)$ be the formula stating that
\begin{center} 
$1_{\collkappa} \Vdash_{\collkappa} `` V_\gamma \models \psi (r, v, x) ''$.
\end{center}
Then $(S, \varphi)$ is a Solovay code for $X$.

  $(2)$ : See \cite{woodin2010axiom} for a proof of that every subset of $\omega^\omega$ is $\infty$-Borel. 
  For every set $S$ of ordinals with $S \in \LR$, 
  $(\omega^\omega)^{L[S]}$ is a  well-orderable subset of $\omega^\omega$ in $\LR$.
  Thus by the perfect set property in $\LR$,
  $(2^\omega)^{L[S]}$ is countable in $\LR$.
  Then it follows that $\omega_1$ is strong limit in $L[S]$.

  Now we assume $\ZF+\DC$, thus $\LR$ satisfies $\ZF+\DC$.
  In particular, $\omega_1$ is a regular cardinal in $\LR$.
  For each set $S$ of ordinals with $S\in \LR$, since $L[S] \subseteq \LR$, $\omega_1$ is a regular cardinal in $L[S]$.
  Hence, $\omega_1$ is inaccessible in $L[S]$.
\end{proof}

 \subsection{($\omega^\omega, \leq^*) $ v.s. ($\mcalM, \subseteq$)}

 Here, we prove that 
 $\dominating \preceq_{pT} \meagerideal$ and 
 $\meagerideal \not\preceq_{pT} \dominating$ under the assumption $(\star)$.

As is well-known, in $\ZFC$, we can construct a Tukey map from $(\omega^\omega, \leq^*)$ to $(\mcalM, \subseteq)$.
That construction works in $\ZF+\DC$.
The detail of the construction of functions in Fact \ref{F_sigma} 
can be found in the proof of Theorem 2.3.1 in \cite{bartoszynski1995set}.

\begin{fact}\label{F_sigma}
  Let $\NWD$ be the set of all nowhere dense subsets of $2^\omega$.
  Then there are functions $\varphi \colon \omega^\omega\rightarrow \mcalM$ and $\tilde{\varphi} \colon \NWD^\omega \rightarrow \omega^\omega$ such that
  \begin{center}
    $\varphi(f)\subseteq \bigcup_{n\in\omega} F_n \Rightarrow f\leq^* \tilde{\varphi}(\langle F_n \mid n < \omega \rangle)$
  \end{center} 
  for any $f \in \omega^\omega$ and any $\langle F_n \mid n < \omega \rangle \in \NWD^\omega$.
\end{fact}

\begin{theorem}\label{ZFC_dom_meager}
 $(\omega^\omega, \leq^*)\preceq_{T} (\mcalM, \subseteq)$. 
\end{theorem}
 
\begin{proof}

  Fix functions $\varphi \colon \omega^\omega\rightarrow \mcalM$ and $\tilde{\varphi} \colon \NWD^\omega \rightarrow \omega^\omega$ as in Fact \ref{F_sigma}.
  We show that $\varphi$ is a Tukey map from $\dominating$ to $\meagerideal$.
  Suppose that $X \in \mcalM$.
  We can take $\langle F_n \mid n<\omega\rangle \in \NWD^\omega$ such that 
  $\bigcup_{n<\omega} F_n = X$.
  Then $\tilde{\varphi}(\langle F_n \mid n<\omega\rangle) \in \omega^\omega$ witnesses that $\varphi$ is a Tukey map.
\end{proof}

 In the rest of this subsection, we prove that $( \mcalM , \subseteq ) \not\preceq_{pT} ( \omega^\omega , \leq^* )$ under the assumption $(\star)$.
 We need some preliminaries.
 Recall that for an inner model $W$, 
 $C^W$ denotes the set of all Cohen reals over $W$.
 
 \begin{lemma}\label{Cohen_cofinal}
  Assume ($\star$).
   Let $x \in \omega^\omega$ and $B \subseteq 2^\omega$ be a comeager set.
   Suppose that $(S, \varphi)$ is a Solovay code for $B$.
   Then $C^{L[S]} \subseteq B$.
 \end{lemma}
 
 \begin{proof}
   For a contradiction, we assume that $C^{L[S]} \not\subseteq B$ and take $c \in C^{L[S]} \setminus B$.
   Since $c$ is a Cohen real over $L[S]$, we can take a $\mbbC$-generic filter $g$ over $L[S]$ such that $c = \dot{c}[g]$, where $\dot{c}$ is the canonical $\mbbC$-name for a generic real.
   We note that
   \begin{align*}
     c \notin B 
     &\Leftrightarrow L[S][c]\models \lnot \varphi(S,c)\\
     &\Leftrightarrow \exists p\in g\ ( L[S]\models  p\Vdash_{\mbbC} \lnot \varphi(S,\dot{c})).
   \end{align*}
   
   Take such $p \in g$, and let $X$ be the collection of $\dot{c}[h]$ for all $\mbbC$-generic filter $h$ over $L[S]$ with $p \in h$.
   Then $X \cap B = \emptyset$, and thus $X$ is meager.
   However, $X$ is comeager in a non-meager set $\{y \in 2^\omega \mid p\subseteq y\}$
   since $(2^\omega)^{L[S]}$ is countable.
   Thus $X$ is non-meager.
   This is a contradiction.
 \end{proof}

 Let $\hat{\mcalM}$ be the comeager filter over $2^\omega$.

 \begin{lemma}\label{Cohen_cofinal_M}
  Assume ($\star$).
  Then $\{C^{L[x]}\mid x\in\omega^\omega\}$ is cofinal in $(\hat{\mcalM}, \supseteq)$.
 \end{lemma}

 \begin{proof}
  Suppose $X\in \hat{\mcalM}$. We can take a Borel comeager set $Y\subseteq X$ and a Borel code $a\in\omega^\omega$ of $Y$.
  By the absoluteness of Borel codes between models of $\ZF+\DC$, there is a formula $\varphi$ such that $( a, \varphi )$ is a Solovay code for $Y$.
  By Lemma \ref{Cohen_cofinal}, $C^{L[a]} \subseteq Y$.
 \end{proof}
 
 \begin{lemma}\label{d_cofM_equiv}
  Assume ($\star$).
  Suppose that for any set $S$ of ordinals,
  there exists $x\in\omega^\omega$ such that for every $y\in\omega^\omega$, there exists $z\in\omega^\omega$ with $y \leq^* z$ and $C^{L[S][z]}\nsubseteq C^{L[S][x]}$.
  Then $\meagerideal\npreceq_{pT}\dominating$.
 \end{lemma}
 
 \begin{proof}
  To prove the contrapositive, assume that $(\mcalM, \subseteq)\preceq_{pT}(\omega^\omega, \leq^*)$.
   Since $\meagerideal$ is isomorphic to 
   $\comeagerfilter$,
   there is a pre-convergent map $\sigma$ from $( \omega^\omega , \leq^* )$ to $( \hat{\mcalM} , \supseteq )$.
   We define $f \colon \omega^\omega \rightarrow \hat{\mcalM}$ by $f(z)=\bigcup \sigma(z)$.
   It is easy to check that $f$ is a convergent map from $( \omega^\omega , \leq^* )$ to $( \hat{\mcalM} , \supseteq )$.
   By $(\star)$ we can take a set $S$ of ordinals and formula $\varphi$ such that
   \begin{center}
   $u\in f(v) \Leftrightarrow L[S][u][v]\models \varphi(S,u,v)$
   \end{center}
   for all $u,v\in\omega^\omega$.
 
   We suppose $x \in \omega^\omega$ and find $y\in\omega^\omega$ such that if $y\leq^* z$ then $C^{L[S][z]}\subseteq C^{L[S][x]}$ for any $z\in\omega^\omega$.
   Since $f$ is convergent and $C^{L[S][x]} \in \hat{\mcalM}$,
   there is $y\in \omega^\omega$ such that 
   $f(z) \subseteq C^{L[S][x]}$ for any $z \in \omega^\omega$ with $y \leq^* z$.
   Note that $C^{L[S][z]} \subseteq f (z)$ 
   for each $z \in \omega^\omega$ by Lemma \ref{Cohen_cofinal}. 
   Thus $y$ is as desired.
 \end{proof}
 
 We will prove that the assumption of Lemma \ref{d_cofM_equiv} holds.
 For this we use the Laver forcing.
 Next we review basics on the Laver forcing.
 Note that we only use forcings over models satisfying $\AC$.
 
 Let $T \subseteq \omega^{<\omega}$ be a tree.
 $\stem (T)$ is the $\subseteq$-maximal $t\in T$ which is comparable with all $s\in T$.
 For $t \in T$ let $\suc_T (t)$ denote the set of all immediate successors of $t$ in $T$.
 $T$ is called a \textit{Laver tree} if $\stem (T)$ exists, and $\suc_T (t)$ is infinite for all $t \in T$ with $\stem (T) \subseteq t$.
 
 The Laver forcing notion $\mbbL$ is the poset of all Laver trees ordered by inclusion.
 The following fact is well-known.
 The proof can be found in \cite{bartoszynski1995set}.
 
 \begin{fact} \label{fact:Laver_basic}
 Assume $\AC$ in $V$.
 Let $H$ be an $\mbbL$-generic filter over $V$.
 Then the following hold in $V[H]:$
 \begin{enumerate}
     \item $\bigcup \bigcap H \in \omega^\omega$, and $x \leq^* \bigcup \bigcap H$ for any $x \in ( \omega^\omega )^V$,
     \item There is no Cohen real over $V$.
 \end{enumerate}
 \end{fact}
 
 Suppose $W$ is an inner model of $\ZFC$. $x \in \omega^\omega$ is called a \textit{Laver real} over $W$ if there is an $\mbbL^W$-generic filter $H$ over $W$ such that $x = \bigcup \bigcap H$.
 Note that $x \in \omega^\omega$ is a Laver real over $W$ if and only if
 \[
 H = \{ T \in \mbbL^W \mid x \in [T] \}
 \]
 is an $\mbbL^W$-generic filter over $W$, where $[T]$ denotes the set of all $x \in \omega^\omega$ such that $x \restrict k \in T$ for all $k < \omega$.
 
 To state other basic properties of $\mbbL$, we give some notation.
 Let $T \subseteq \omega^{< \omega}$ be a tree.
 For $t \in T$ let $T \restrict t$ be the set of all elements of $T$ comparable with $t$.
 
 For $S, T\in \mbbL$ and $n < \omega$, we write $S \leq_n T$ if $S \subseteq T$, $\stem (S) = \stem (T)$, and $S \cap \omega^{m+n} = T \cap \omega^{m+n}$, where $m = \dom ( \stem (S) )$.
 For $S \in \mbbL$ and $D \subseteq \mbbL$, we say that $S$ \textit{meets} $D$ if there is $T \in D$ with $S \subseteq T$.
 
 A \textit{front} of $T$ is a pairwise incomparable subset of $T$ such that for any $x \in [T]$ there is $k < \omega$ with $x \restrict k \in A$.
 Note that a pairwise incomparable $A \subseteq T$ is a front of $T$ if and only if
 \[
 S = \{ s \in T \mid \forall k \leq \dom (s) \ ( s \restrict k \notin A ) \}
 \]
 is (conversely) well-founded. So the statement ``$A$ is a front of $T$'' is absolute between transitive models of $\ZF$.
 
 
 The following lemma is standard. The proof for the case when $n = 0$ can be found in \cite[1.9 Fact]{judah1990kunen}.
 We note that the argument works in our base setting $\ZF+\DC$.
 
 \begin{lemma}\label{lem:front}
   Suppose $T \in \mbbL$, $n < \omega$ and $D\subseteq \mbbL$ is predense.
   Then there are $S \leq_n T$ and a front $A$ of $S$ such that $S_a$ meets $D$ for all $a\in A$.
 \end{lemma}
 
 \begin{proof}
 See \cite[1.9 Fact]{judah1990kunen} for the proof of the case when $n=0$.
 Using the case when $n=0$, we prove the case when $n>0$.
 
 Let $m = \dom ( \stem (T) )$ and $B = T \cap \omega^{m+n}$.
 By the case when $n=0$, for each $t \in B$, we can take $S_t \leq_0 T \restrict t$ and a front $A_t$ of $S_t$ such that $S_t \restrict a \in D$ for all $a \in A_t$.
 Then it is easy to check that $S = \bigcup_{t \in B} S_t$ and $A = \bigcup_{t \in B} A_t$ are as desired.
 \end{proof}

 Suppose $W$ is an inner model of $\ZFC$, and $T \in \mbbL$.
 We say call $T$ an $\mbbL^W$-\textit{generic condition} over $W$ if for any predense $D \subseteq \mbbL^W$ with $D \in W$, there is a front $A$ of $T$ such that $T \restrict a$ meets $D$ for any $a \in A$.
 Suppose $T$ is an $\mbbL^W$-generic condition over $W$.
 Then every predense $D \subseteq \mbbL^W$ with $D \in W$ is predense below $T$ in $\mbbL$.
 Moreover, every $x \in [T]$ is a Laver real over $W$.
 
 The following lemma is almost the same as the properness of the Laver forcing.

 \begin{lemma}\label{Laver_generic}
   Suppose $W$ is an inner model of $\ZFC$ with $( 2^{2^\omega} )^W$ is countable. Then for any $T \in \mbbL^W$ there is an $\mbbL^W$-generic condition $S \in \mbbL$ over $W$ such that $S \subseteq T$.
 \end{lemma}
 
 \begin{proof}
   Let $\langle D_n\mid n<\omega \rangle$ be an enumeration of predense subsets of $\mbbL^W$ in $W$.
   By Lemma \ref{lem:front}, we can recursively construct a descending sequence $\langle T_n \mid n < \omega \rangle$ in $\mbbL^W$ below $T$ together with a sequence $\langle A_n \mid n < \omega \rangle$ such that $A_n$ is a front of $T_n$, and $T_n \restrict a$ meets $D_n$ for all $a \in A_n$.
 
   Let $S = \bigcap_{n \in \omega} T_n \in \mbbL$.
   Then $S \subseteq T$.
   Moreover, for each $n < \omega$, it is easy to check that $A_n \cap S$ is a front of $S$, and $S \restrict a$ meets $D_n$.
   So $S$ is an $\mbbL^W$-generic condition over $W$.
 \end{proof}
 
 
 
 \begin{theorem}
   Assume $(\star)$.
   Then $(\mcalM, \subseteq) \npreceq_{pT} (\omega^\omega, \leq^*)$.
 \end{theorem}
 
 \begin{proof}
  Suppose $S$ is a set of ordinals.
  By Lemma \ref{d_cofM_equiv}, it suffices to find $x \in \omega^\omega$ such that for all $y \in \omega^\omega$ there is $z \in \omega^\omega$ with $y \leq^* z$ and $C^{L[S][z]} \not\subseteq C^{L[S][x]}$.
 
   We can take $x \in \omega^\omega$ such that $( 2^{2^\omega} )^{L[S]}$ is countable in $L[S][x]$.
   Suppose $y \in \omega^\omega$.
   We find $z \geq^* y$ with $C^{L[S][z]} \not\subseteq C^{L[S][x]}$.
   In $L[S][x]$, by Lemma \ref{Laver_generic}, we can take an $\mbbL^{L[S]}$-generic condition $T \in \mbbL^{L[S][x]}$ over $L[S]$.
   Note that $T$ is an $\mbbL^{L[S]}$-generic condition over $L[S]$ in $V$, too.
   Moreover, again by Lemma \ref{Laver_generic}, take an $\mbbL^{L[S][x]}$-generic condition $T' \in \mbbL$ over $L[S][x]$ with $T' \subseteq T$.
   Since $T'$ is infinitely branching above its stem, we can take $z \in [T']$ with $y \leq^* z$.
   We show that $C^{L[S][z]} \not\subseteq C^{L[S][x]}$.
   Note that $z$ is a Laver real over both of $L[S]$ and $L[S][x]$.
   
   In $L[S][z]$ let $\Gamma$ be the collection of Borel codes of meager subsets of $2^\omega$.
   Note that $( 2^\omega )^{L[S][z]} \leq ( 2^{2^\omega} )^{L[S]}$ since $z$ is a Laver real over $L[S]$.
   So $\Gamma$ is countable in $L[S][x][z]$.
   Then $B = \bigcup_{c\in \Gamma} B_c^{L[S][x][z]}$ is meager in $L[S][x][z]$, where $B^{L[S][x][z]}_c$ is the Borel meager set coded by $c$ in $L[S][x][z]$.
   So we can take $w\in (\omega^\omega)^{L[S][x][z]} \setminus B$.
   Then $w \in C^{L[S][z]}$.
   On the other hand, $w$ is in $L[S][x][z]$, a Laver extension of $L[S][x]$.
   Thus $w \notin C^{L[S][x]}$ by Fact \ref{fact:Laver_basic} (2).
 \end{proof}
 \subsection{($\mcalM, \subseteq$) v.s. ($\mcalN, \subseteq$)}
  
 In this subsection, we prove the pre-Tukey relations 
 $\meagerideal \preceq_{pT} \nullideal$ and 
 $\nullideal \npreceq_{pT} \meagerideal$ under $(\star)$.

 In $\ZFC$ we can construct a Tukey map from 
 $\meagerideal$ to $\nullideal$.
 The construction can be found in the proof of  Theorem 2.3.1 in \cite{bartoszynski1995set}.
 In the construction, Fact \ref{ZFC_meager_null} below is proved in $\ZF+\DC$.
 To construct a Tukey map using Fact \ref{ZFC_meager_null}, it is necessary to choose a countable sequence of nowhere dense sets corresponding to each meager set.
 On the other hand, by considering the collection of all such sequences, we can define a pre-Tukey map in $\ZF+\DC$, without using $\AC$.

\begin{fact}\label{ZFC_meager_null}
  Let $\cNWD$ be the set of all closed nowhere dense subsets of $2^\omega$ and let $\mathrm{OPEN}$ be the set of all open subsets of $2^\omega$.
  Then there are functions 
  $\varphi \colon {\cNWD}^\omega \rightarrow \mcalN$ and $\tilde{\varphi} \colon \mathrm{OPEN}^\omega \rightarrow \mcalM$ such that
  \begin{center}
    $\varphi (\langle F_n \mid n<\omega \rangle) \subseteq \bigcap_{n<\omega} G_n 
    \Rightarrow \bigcup_{n<\omega} F_n \subseteq \tilde{\varphi}(\langle G_n \mid n<\omega \rangle)$
  \end{center}
  for any $\langle F_n \mid n<\omega \rangle \in \cNWD^\omega$ and any 
  $\langle G_n \mid n<\omega \rangle \in \mathrm{OPEN}^\omega$ with $\bigcap_{n<\omega} G_n \in \mcalN$.
\end{fact}

\begin{theorem}
$(\mcalM, \subseteq)\preceq_{pT} (\mcalN, \subseteq)$. 
\end{theorem}

\begin{proof}
 Let $\mcalM_\sigma$ be the set of all $F_\sigma$ meager subsets of $2^\omega$.
 Since $\mcalM_\sigma$ is cofinal in $\meagerideal$,
 it suffices to show that 
 $( \mcalM_\sigma, \subseteq ) \preceq_{pT} \nullideal$.
 Fix functions 
  $\varphi \colon {\cNWD}^\omega \rightarrow \mcalN$ and $\tilde{\varphi} \colon \mathrm{OPEN}^\omega \rightarrow \mcalM$ as in Fact \ref{ZFC_meager_null}.
  We define
  $\pi \colon M_\sigma \rightarrow \mcalP(\mcalN)$ by 
  \begin{center}
  $\pi(Z)=\{\varphi(\langle Z_n \mid n<\omega \rangle) \mid \langle Z_n \mid n<\omega \rangle \in \cNWD^\omega \land Z=\bigcup_{n<\omega} Z_n \}.$
  \end{center}
  We show that $\pi$ is a pre-Tukey map.
  Suppose $X \in \mcalN$ and find $Y \in M_\sigma$ such that if $\pi(Z) \cap X\downarrow_{\subseteq}$ is non-empty then $Z \subseteq Y$, for any $Z \in \mcalM_\sigma$.
  We take $\langle G_n \mid n<\omega \rangle \in \mathrm{OPEN}^\omega$ such that $X \subseteq \bigcap_{n<\omega} G_n \in \mcalN$.
  We can take $Y \in M_\sigma$ such that $\tilde{\varphi}(\langle G_n \mid n<\omega \rangle) \subseteq Y$.
  By Fact \ref{ZFC_meager_null}, $Y$ is as desired.
\end{proof}

Next, we show $(\mcalN, \subseteq)\npreceq_{pT} (\mcalM, \subseteq)$ under the assumption $(\star)$.
To prove this,
we note the fact that the two-step iteration of Cohen forcing and Hechler forcing adds a meager set covering all meager sets coded in the ground model.

 The Hechler forcing notion $(\mbbD, \leq_\mbbD)$ is defined as follows:
 \begin{center}
 $\mbbD=\{(s,f)\mid s\in\omega^{<\omega}$, $f\in\omega^\omega$ and $s\subseteq f\}$,
 \end{center}
 and $(t,g)\leq_{\mbbD} (s,f)$ if $t\supseteq s$, $g(n)\geq f(n)$ for all $n\in\omega$ 
 and $t(i)\geq f(i)$ for all $i\in |t|\setminus |s|$.

We henceforth fix a coding of elements of $2^\omega \times \omega^\omega$ by elements of $\omega^\omega$.

\begin{definition}
 We say that $x\in \omega^\omega$ is \textit{Cohen-Hechler real} over $V$ if
 $x$ codes the pair $(c,d)\in 2^\omega\times \omega^\omega$ such that $c$ is Cohen over $V$
 and $d$ is Hechler over $V[c]$. 
\end{definition}
Note that for models $W\subseteq V$ of $\ZFC$  if $x\in\omega^\omega$ is
Cohen-Hechler real over $V$ then Cohen-Hechler over $W$.

The proof of the following fact can be found in \cite{bartoszynski2005hechler}.

\begin{fact}\label{covering_meager}
  Assume $\AC$ in $V$.
 Let $c\in 2^\omega$ be a Cohen real over $V$ and let $d\in\omega^\omega$ be a Hechler real over $V[c]$.
 Then there is a meager set coded in $V[c][d]$ which covers all meager sets coded in $V$.
 Moreover, a Borel code of such a meager set can be explicitly defined from $c$ and $d$.
\end{fact}

\begin{lemma}\label{Cohen_Hechler}
Assume $\AC$ in $V$.
 Suppose that $y\in\omega^\omega$ and
 $z\in \omega^\omega$ is a Cohen-Hechler real over $V[y]$.
 Then $C^{V[z]} \subseteq C^{V[y]}$.
\end{lemma}

\begin{proof}
 Let $(c,d)\in 2^\omega\times \omega^\omega$ be the pair of the Cohen real over $V[y]$ and the Hechler real over $V[y][c]$ coded by $z$.
 Since $z$ is Cohen-Hechler real over $V$,
 by Fact \ref{covering_meager} there is a meager set $X$ coded in $V[z]$ which covers all meager sets coded in $V$.
 Since $z$ is Cohen-Hechler over $V[y]$, $X$ covers all meager sets coded in $V[y]$ too.
 If $r$ is Cohen over $V[z]$ then $r\notin X$,
 so $r$ is not contained by any meager set coded in $V[y]$.
 Thus $r$ is Cohen over $V[y]$.
\end{proof}

We can show the following lemmas in the same way as the proof of 
Lemma \ref{Cohen_cofinal} and Lemma \ref{d_cofM_equiv}.
Let $\hat{\mcalN}$ be the conull filter over $2^\omega$.

\begin{lemma}\label{random_cofinal}
  Assume ($\star$).
   Let $x \in \omega^\omega$ and $B \subseteq 2^\omega$ be a conull set.
   Suppose that $(S, \varphi)$ is a Solovay code for $B$.
   Then $R^{L[S]} \subseteq B$.
 \end{lemma}

 \begin{lemma}\label{random_cofinal_M}
  Assume ($\star$).
  Then $\{R^{L[x]}\mid x\in\omega^\omega\}$ is cofinal in $(\hat{\mcalN}, \supseteq)$.
 \end{lemma}

\begin{lemma}\label{cofM_cofN_equiv}
 Assume $(\star)$.
 Suppose that for any set $S$  of ordinals,
 there exists $x\in\omega^\omega$ such that for every $y\in\omega^\omega$ 
 there exists $z\in\omega^\omega$ with $C^{L[z]}\subseteq C^{L[y]}$ and $R^{L[S][z]}\nsubseteq R^{L[S][x]}$.
 Then $(\mcalN, \subseteq)\npreceq_{pT} (\mcalM, \subseteq)$.
\end{lemma}

\begin{proof}
  To prove the contrapositive, assume that $(\mcalN, \subseteq)\preceq_{pT}(\mcalM, \subseteq)$.
  By Lemma \ref{Cohen_cofinal_M} we obtain a pre-convergent map $\sigma$
  from $(\{C^{L[x]}\mid x\in\omega^\omega\}, \supseteq)$ to $(\hat{\mcalN}, \supseteq)$.
  We define $f(C^{L[x]})=\bigcup \sigma(C^{L[x]})$ for each $x\in\omega^\omega$, 
  then $f$ is a convergent map from $(\{C^{L[x]}\mid x\in\omega^\omega\}, \supseteq)$ to $(\hat{\mcalN}, \supseteq)$.
  By $(\star)$ we can take 
   a set $S$ of ordinals and formula $\varphi$ such that 
   \begin{center}
    $u\in f(C^{L[v]}) \Leftrightarrow L[S][u][v]\models \varphi(S,u,v)$
   \end{center}
   for all $u,v\in \omega^\omega$.

   We suppose $x\in\omega^\omega$ and find $y\in\omega^\omega$ such that for every $z\in\omega^\omega$ if $C^{L[z]}\subseteq C^{L[y]}$ then  $R^{L[S][z]} \subseteq R^{L[S][x]}$.
   Since $f$ is convergent and $R^{L[S][x]} \in \hat{\mcalN}$,
   there is $y\in \omega^\omega$ such that 
   $f(C^{L[z]}) \subseteq R^{L[S][x]}$ for any $z\in\omega^\omega$ with $C^{L[z]} \subseteq C^{L[y]}$.
   By Lemma \ref{random_cofinal}, $R^{L[S][z]} \subseteq f(C^{L[z]})$.
   Thus if $C^{L[z]} \subseteq C^{L[y]}$ then 
   $R^{L[S][z]}\subseteq R^{L[S, x]}$.
\end{proof}

\begin{theorem}
 Assume $(\star)$. Then $(\mcalN, \subseteq) \npreceq_{pT} (\mcalM, \subseteq)$.
\end{theorem}

\begin{proof}
  Suppose $S$ is a set of ordinals.
 By Lemma \ref{cofM_cofN_equiv}, it suffices to find
 $x\in \omega^\omega$ such that for all $y\in \omega^\omega$ there is $z\in \omega^\omega$ with $C^{L[z]} \subseteq C^{L[y]}$ and $R^{L[S][z]} \nsubseteq R^{L[S][x]}$.

  We can take $x\in\omega^\omega$ such that $(2^{2^\omega})^{L[S]}$ is countable in $L[S][x]$.
  We suppose $y\in\omega^\omega$ and find $z\in\omega^\omega$ such that $C^{L[z]} \subseteq C^{L[y]}$ and $R^{L[S][z]} \nsubseteq R^{L[S][x]}$.
   Let $(c,d)$ be a $\mathbb{C}\ast\dot{\mathbb{D}}$-generic real over $L[S][x][y]$,
   and let $z\in\omega^\omega$ be a Cohen-Hechler real coding $(c,d)$.
   Note that $z$ is Cohen-Hechler over $L[y]$ too.
   In particular $C^{L[z]} \subseteq C^{L[y]}$ by Lemma \ref{Cohen_Hechler}.
   We show that $R^{L[S][z]} \nsubseteq R^{L[S][x]}$.

   In $L[S][z]$ let $\Gamma$ be the collection of Borel codes of null subsets of $2^\omega$.
   Since $\CD$ is c.c.c, $(2^\omega)^{L[S][z]}$ is countable in $L[S][x][z]$.
   Thus $\Gamma$ is countable in $L[S][x][z]$, so $B=\bigcup_{c\in\Gamma} B_c^{L[S][x][z]}$ is null in $L[S][x][z]$.
   Hence we can take $w\in (\omega^\omega)^{L[S][x][z]}\setminus B$. Then $w$ is random over $L[S][z]$.
 
   On the other hand, $w$ is in $L[S][x][z]$, a $\CD$ extension of $L[S][x]$.
   Since $\CD$ is $\sigma$-centered and  $\sigma$-centered forcing does not add
   random reals, $w$ is not random over $L[S][x]$.
\end{proof}

 \subsection{$(\omega_1 , \leq)$}
 In this subsection, we show that $\dominating \npreceq_{pT} \omegaone$ and $\omegaone \npreceq_{pT} \nullideal$.
 From these results and the results in the previous subsections, it follows that 
 $\omegaone \npreceq_{pT} (D, \leq_D)$ and $(D, \leq_D) \npreceq_{pT} \omegaone$ for any $(D, \leq_D)= \dominating, \meagerideal, \nullideal$.
 \begin{theorem}
  Assume $(\star)$.
  Then $\dominating \npreceq_{pT} \omegaone$.
 \end{theorem}

 \begin{proof}
  For $f, g\in \omega^\omega$, we write $f <^* g$ if $f(n) < g(n)$ for all but finitely many $n\in\omega$.
  Since $(\omega^\omega, <^*)$ is Tukey reducible to $\dominating$ and $\omegaone^*$ is isomorphic to $\omegaone$,
  it suffices to show that 
  $(\omega^\omega, <^*) \npreceq_{T} \omegaone$.
  We suppose $f$ is a function from $\omega^\omega$ to $\omega_1$ and show that $f$ is not a Tukey map.
  Let
  \begin{center}
     $X_f=\{(x,z)\in\omega^\omega\times \omega^\omega \mid 
     z \in \WO \land f(x)=o(z)\}$.
  \end{center}
  By $(\star)$ we can take a set $S$ of ordinals  and a formula $\varphi$ such that
  \begin{center}
    $(x,z) \in X_f \Leftrightarrow
    L[S][x][z] \models \varphi(S, x, z)$.
  \end{center}
  Let $c\in\omega^\omega$ be a Cohen real over $L[S]$ and let $\alpha = f(c)$.
  We show that for all $x\in\omega^\omega$ there exists $y\in\omega^\omega$ such that $f(y)=\alpha$ and $y \nless^* x$.

  Suppose $x\in\omega^\omega$.
  Let $\dot{z}$ be a $\collalpha$-name for a real and let $h$ be a $\collalpha$-generic filter over $L[S][x]$ such that $\dot{z}[h] \in \WO$ and $o(\dot{z}[h])=\alpha$.
  Then $\varphi(S, c, \dot{z}[h])$ holds in $L[S, c, \dot{z}[h]]$.
  We can take $p\in\mbbC$ and $q\in \collalpha$ such that 
  \begin{center}
    $L[S] \models p \Vdash_{\mbbC} q \Vdash_{\collalpha} \varphi (S, \dot{c}, \dot{z})$
  \end{center}
  where $\dot{c}$ is the canonical $\mbbC$-name for a generic real.
  We take a $\mbbC$-generic filter $g$ over $L[S][x]$ with $p \in g$ and let $y=\dot{c}_g$.
  We take a $\collalpha$-generic filter $h'$ over $L[S][x][y]$ with $q \in h'$ and
  let $z' = \dot{z}[h']$. Then $\varphi(S, y, z')$ holds in $L[S][y][z']$. Thus $f(y)=\alpha$.
  Since $y$ is a Cohen real over $L[S][x]$ and each Cohen real takes the same value infinitely often as any real in the ground model, 
  $y \nless^* x$ holds.
 \end{proof}

 Next, to prove $\omegaone \npreceq_{pT} \nullideal$
 we review basics on the amoeba forcing, a forcing which adds a null set covering all null sets coded in the ground model.

 The amoeba forcing notion is 
 \begin{center}
  $\mbbA=\{U \subseteq 2^\omega \mid U$ is open and $\mu(U)<1/2\}$
 \end{center}
 ordered by reverse inclusion.
 \begin{fact}\label{fact_amoeba_covering}
  Assume $\AC$ in $V$.
  Let $G$ be an $\mbbA$-generic filter over $V$.
  Then $G$ induces a Borel code of $G_\delta$ null set which covers all null sets coded in $V$.
  In particular, $R^{V[G]} \subseteq R^V$.
 \end{fact}
 Suppose  $W$ is an inner model of $\ZFC$.
 $a\in\omega^\omega$ is called an \textit{amoeba real} over $W$ if $a$ is a Borel code of a null set as in Fact \ref{fact_amoeba_covering} induced by an $\mbbA$-generic filter over $W$.
 Just as the meager ideal corresponds to Cohen forcing and the null ideal corresponds to random forcing, there is a $\sigma$-ideal associated with  amoeba forcing. 
 See 3.4.B in \cite{bartoszynski1995set} for the definition of the ideal and basics.
 From this characterization, we can obtain the following.
 \begin{fact}\label{fact_amoeba_generic}
  Suppose $W \subseteq V$ are transitive models of $\ZFC$.
  Let $G$ be an $\mbbA$-generic filter over $V$ and
  let $a\in\omega^\omega$ be an amoeba real induced by $G$. 
  Then there is an $\mbbA$-generic filter $G'$ over $W$ such that
  \begin{enumerate}
    \item $G'$ also induces $a$. Thus $a$ is an amoeba real also over $W$.
    \item For any $p \in \mbbA^W$ there is $q \in \mbbA^V$ such that $p \in G'$ if and only if $q\in G$.
  \end{enumerate}
 \end{fact}

 \begin{theorem}
    Assume $(\star)$.
    Then $\omegaone \npreceq_{pT} \nullideal$.
 \end{theorem}
 \begin{proof}
  If there is a pre-convergent map from $\nullideal$ to $\omegaone$ then we can easily construct a convergent map from $(\{R^{L[x]} \mid x\in \omega^\omega\}, \supseteq)$ to $(\omega_1, \leq)$.
  Thus we suppose $f$ is a function from $\{R^{L[x]} \mid x\in \omega^\omega\}$ to $\omega_1$
  and show that $f$ is not convergent.
  Let
  \begin{center}
     $X_f=\{(x,z)\in\omega^\omega \times \omega^\omega \mid 
     z \in \WO \land f(R^{L[x]})=o(z)\}$.
  \end{center}
  By $(\star)$ we can take a set $S$ of ordinals and a formula $\varphi$ such that
  \begin{center}
    $(x,z) \in X_f \Leftrightarrow
    L[S][x][z] \models \varphi(S, x, z)$.
  \end{center}
  Let $a\in\omega^\omega$ is an amoeba real over $L[S]$ and suppose that $f(R^{L[a]})=\alpha$.
  We show that for all $x\in\omega^\omega$ there exists $y\in\omega^\omega$ such that $R^{L[y]}\subseteq R^{L[x]}$ and $f(R^{L[y]})=\alpha$.

  Suppose $x\in\omega^\omega$.
  Let $\dot{z}$ be $\collalpha$-name for a real and let $h$ be a $\collalpha$-generic filter over $L[S][x]$ such that $\dot{z}[h] \in \WO$ and $o(\dot{z}[h])=\alpha$. Let $z=\dot{z}[h]$.
  Then $\varphi(S, a, z)$ holds in $L[S, a, z]$.
  We can take $p\in\mbbA^{L[S]}$ and $q\in \collalpha$ such that
  \begin{center}
    $L[S] \models p\Vdash_{\mbbA} q\Vdash_{\collalpha} \varphi(S, \dot{a}, \dot{z})$
  \end{center}
  where $\dot{a}$ is the canonical $\mbbA$-name for a generic real.
  We take an $\mbbA^{L[S][x]}$-generic filter $g$ over $L[S][x]$ with $p \in g$ and let $y=\dot{a}[g]$.
  We take a $\collalpha$-generic filter $h'$ over $L[S][x][y]$ with $q \in h'$
  and let $z'=\dot{z}[h']$.
  Then $\varphi(S, y, z')$ holds in $L[S][y][z']$. Thus $f(R^{L[y]})=\alpha$.
  By Fact \ref{fact_amoeba_covering}, \ref{fact_amoeba_generic} and that $y$ is an amoeba real over $L[S][x]$, 
  $R^{L[y]}\subseteq R^{L[x]}$ holds.
 \end{proof}

 \subsection{$([\omega^\omega]^\omega, \subseteq)$}

 First, we show that $([\omega^\omega]^\omega, \subseteq)$ is pre-Tukey equivalent to the  constructibility degree on reals under the assumption $(\star)$.

 We define an order $\leq_L$ on $\omega^\omega$ by
 \begin{center}
    $x\leq_L y \Leftrightarrow x\in L[y]$.
 \end{center}
 \begin{lemma}\label{L-relation}
  Assume $(\star)$.
  Then $(\omega^\omega, \leq_L)\equiv_{pT} ([\omega^\omega]^\omega, \subseteq)$.
 \end{lemma}
 \begin{proof}
  We define $f \colon \omega^\omega \rightarrow [\omega^\omega]^\omega$  by $f(x)=(\omega^\omega)^{L[x]}$. 
  For $X\in [\omega^\omega]^\omega$ we can take $r_X \in\omega^\omega$ which codes a bijection from $\omega$ to $X$.
  We can show straightforwardly that for each $X \in [\omega^\omega]^\omega$, $r_X$ witnesses that $f$ is Tukey and convergent.
 \end{proof}

 The following facts are standard results in forcing theory. 
 These proofs can be found in Lemma 15.43 in \cite{jech2006set} and Theorem A.0.9 in \cite{larson2004stationary}.

 \begin{fact}\label{fact:immediate}
  Assume $\AC$ in $V$.
  Let $B$ be a complete Boolean algebra and 
  let $G$ be a $B$-generic filter over $V$.
  If $M$ is a model of $\ZFC$ such that $V \subseteq M \subseteq V[G]$, then there exists a complete subalgebra $D\subseteq B$ such that $M=V[D\cap G]$.
 \end{fact}

 \begin{fact}\label{fact:collapse_assorb}
  Assume $\AC$ in $V$.
  Let $\mbbP$ be a forcing notion with $\mbbP \in V_\alpha$
  and let $g$ be a $\mbbP$-generic filter over $V$.
  Suppose $G$ is a $\collapse{2^\alpha}$-generic filter over $V$.
  Then there is a $\collapse{2^\alpha}$-generic filter $H\in V[G]$ over $V[g]$ such that
  $V[G]=V[g][H]$.
 \end{fact}

 \begin{lemma}\label{lemma:countable_subset}
  Assume $(\star)$.
  Suppose $(\omega^\omega, \sqsubseteq)$ is  $\sigma$-directed.
  Then $(\omega^\omega, \sqsubseteq) \preceq_{pT} ([\omega^\omega]^\omega, \subseteq)$.
 \end{lemma}
 \begin{proof}
  We define $\pi \colon [\omega^\omega]^\omega \rightarrow \mcalP(\omega^\omega)$ By
  \begin{center}
    $\pi(X)=\{x\in\omega^\omega \mid \forall y\in X \  
    (y \sqsubseteq x)\}$.
  \end{center}
  For $x\in\omega^\omega$, each $X\in[\omega^\omega]^\omega$ with $x\in X$ witnesses that $\pi$ is pre-convergent.
 \end{proof}

 \begin{theorem}
  Assume $(\star)$.
  \begin{enumerate}
    \item $(\mcalN, \subseteq)\preceq_{pT} ([\omega^\omega]^\omega, \subseteq)$.
    \item $(\omega_1, \leq) \preceq_{pT} ([\omega^\omega]^\omega, \subseteq)$.
    \item $([\omega^\omega]^\omega, \subseteq) \npreceq_{pT} (\mcalN, \subseteq)$.
  \end{enumerate}
 \end{theorem}
 \begin{proof}
  Using Lemma \ref{lemma:countable_subset},
  $(1)$ follows from Lemma \ref{random_cofinal_M} and $(2)$ follows from $\omegaone \equiv_{pT} (\WO, \trianglelefteq)$.

  $(3) \colon $
  By Lemma \ref{random_cofinal_M} and \ref{L-relation}, it suffices to show that there is no pre-convergent map from 
  $(\{R^{L[x]} \mid x \in \omega^\omega\}, \supseteq)$ to $(\omega^\omega, \leq_L)$.
  Suppose $\pi$ is a function from 
  $\{R^{L[x]} \mid  x\in\omega^\omega\}$ to
  $\mcalP(\omega^\omega)$.
  Let
  \begin{center}
    $X=\{ (z,r)\in\omega^\omega\times \omega^\omega
    \mid r\in\WO \land
    \exists w \in \pi(R^{L[z]})(o(r)$ is uncountable in $L[w])\}$.
  \end{center}
  We take a set $S$ of ordinals and a formula $\varphi$ such that
  \begin{center}
    $(z, r)\in X \Leftrightarrow
    L[S][z][r] \models \varphi(S,z,r)$.
  \end{center}
  Let $z'\in\omega^\omega$ be an amoeba real over $L[S]$ and take 
  $w' \in \pi(R^{L[z']})$.
  Let $\alpha=\omega_1^{L[w']}$.
  Note that if $r\in \WO$ and $o(r)=\alpha$ then $(z', r) \in X$.

  Let $\delta$ be an ordinal such that $(2^{2^\alpha})^{L[S]} < \delta <\omega_1$.
  Let $\psi(S,z)$ denote the formula 
  \begin{center}
    $\forall r\in \WO (o(r)=\alpha \Rightarrow 
    L[S][z][r]\models \varphi(S, z, r))$.
  \end{center}
  Then by the homogeneity of $\collapse{\delta}$ it follows that
  \begin{center}
    $L[S][z']\models 1_{\collapse{\delta}} \Vdash_{\collapse{\delta}} 
    \psi(S, z', r)$.
  \end{center}
  Thus we can take $p\in\mbbA^{L[S]}$ such that 
  \begin{center}
    $L[S]\models p \Vdash_{\mbbA^{L[S]}} 1_{\collapse{\delta}} \Vdash_{\collapse{\delta}} 
    \psi(S, z', r)$.
  \end{center}
  Let $\dot{x}$ be a $\collapse{\alpha}$-name for a real and 
  let $g$ be a $\collapse{\alpha}$-generic filter over $L[S]$ such that $\dot{x}[g] \in \WO$ and $o(\dot{x}[g])=\alpha$.
  Let $x=\dot{x}[g]$.
  Suppose $y\in\omega^\omega$.
  To show that $\pi$ is not pre-convergent, we find $z,w \in \omega^\omega$ such that $R^{L[z]}\subseteq R^{L[y]}$, $w\in\pi(R^{L[y]})$ and $x \notin L[w]$.

  Let $h$ be an $\mbbA$-generic filter over $L[S][x][y]$ with $p\in h$ and let $z\in\omega^\omega$ be an amoeba real over $L[S][x][y]$ induced by $h$.
  Then $R^{L[z]} \subseteq R^{L[y]}$ holds by Fact \ref{fact_amoeba_covering}.
  Note that $(x, z)\in\omega^\omega \times \omega^\omega$ is $\collapse{\alpha}\ast \dot{\mbbA}$-generic over $L[S]$.
  Let $B(\collapse{\alpha}\ast\dot{\mbbA})$ denote the completion of $\collapse{\alpha}\ast\dot{\mbbA}$.
  $L[S][z]$ is an $\mbbA$-generic extension of $L[S]$, and is an immediate model between $L[S]$ and a $B(\collapse{\alpha}\ast\dot{\mbbA})$-generic extension $L[S][z][x]$.
  Hence by Fact \ref{fact:immediate}, there exists a complete subalgebra $D$ of $B(\collapse{\alpha}\ast\dot{\mbbA})$ such that $L[S][z]$ is a $D$-generic extension of $L[S]$.
  Let $\mbbQ$ be the quotient of $B(\collapse{\alpha}\ast\dot{\mbbA})$ by $D$. Then $L[S][z][x]$ is a $\mbbQ$-generic extension of $L[S][z]$.
  Since the size of $B(\collapse{\alpha}\ast\dot{\mbbA})$ and $\mbbQ$ is less that $\delta$,
  by Fact \ref{fact:collapse_assorb}
  we can take a $\collapse{\delta}$-generic filter $k$ over $L[S][z]$ such that $x \in L[S][z][k]$.
  
  By construction, $\psi(S,z)$ holds in $L[S][z][k]$.
  Thus
  \begin{center}
    $L[S][z][x] \models \varphi(S, z, x)$.
  \end{center}
  This is equivalent to $(z, x) \in X$, so 
  there is $w \in \pi(R^{L[z]})$ such that $\alpha$ is uncountable in $L[w]$.
  Since $o(x)=\alpha$, $x \notin L[w]$ holds.
 \end{proof}

 \subsection{Forcing absoluteness argument}

 From the previous results and Lemma \ref{fact_Solovay_LR}, it follows as a corollary that the pre-Tukey relation $(\mcalN, \subseteq) \npreceq_{pT} (\mcalM, \subseteq) \npreceq_{pT} (\omega^\omega, \leq^*)$  in  the Solovay model and in $L(\mathbb{R})$ with determinacy. 
 On the other hand, these results concerning such models can also be proved (under suitable assumptions) using arguments on the forcing absoluteness. 
 Author learned this argument from Schilhan.

 First, we see the argument for the Solovay model.
 A proof of the following fact can be found in \cite{fuchino2002kunen}.

 \begin{fact}[Kunen]\label{absolute_LR}
  Assume $\AC$ in $V$.
  Suppose that $\kappa$ is a weakly compact cardinal in $V$ and $G$ is a $\collkappa$-generic filter over $V$.
  Let $\mbbP$ be a c.c.c. forcing notion in $V[G]$ and let $H$ be a $\mbbP$-generic filter over $V[G]$.
  Then $\VRVG$ and $V(\mbbR^{V[G][H]})$ are elementary equivalent.
  \end{fact}

 \begin{prop}
  Assume $\AC$ in $V$.
  Suppose that $\kappa$ is a weakly compact cardinal in $V$ and $G$ is a $\collkappa$-generic filter over $V$.
  Then $\VRVG \models (\mcalN, \subseteq) \npreceq_{pT} (\mcalM, \subseteq) \npreceq_{pT} (\omega^\omega, \leq^*)$.
 \end{prop}

 \begin{proof}
  We show that $\VRVG \models (\mcalM, \subseteq) \npreceq_{pT} (\omega^\omega, \leq^*)$.
  Note that $V[G]\models \CH$.
Let $H$ be a $\mbbB_{\omega_2}$-generic filter over $V[G]$, where $\mbbB_{\omega_2}$ is a c.c.c. forcing notion which adds $\omega_2$-many random reals.
Then $V[G][H] \models \mathfrak{d}<\cof(\mcalM)$, so
there is no pre-Tukey map from $(\mcalM, \subseteq)$ to $(\omega^\omega, \leq^*)$ in $V[G][H]$ and $V(\mbbR^{V[G][H]})$.
By Fact \ref{absolute_LR}, $\VRVG \models (\mcalM, \subseteq) \npreceq_{pT} (\omega^\omega, \leq^*)$.

We can show that $\VRVG \models (\mcalN, \subseteq) \npreceq_{pT} (\mcalM, \subseteq)$ in the same way, by using a c.c.c. forcing that forces $\cof(\mcalM)<\cof(\mcalN)$ instead of $\mbbB_{\omega_2}$.
 \end{proof}

 For $\LR$, we use the set forcing absoluteness of the theory of $\LR$ under large cardinal assumptions (Theorem 2.31 in \cite{woodin2010axiom}).
 Note that that large cardinal assumption implies $\AD$ in $\LR$.

 \begin{fact}[Woodin]
  Assume $\AC$ in $V$.
  Suppose that there exist proper class many Woodin cardinals. Let $\mbbP$ be a forcing notion and let $G$ be a $\mbbP$-generic filter over $V$.
  Then $\LR$ and $L(\mbbR^{V[G]})$ are elementarily equivalent. 
 \end{fact}

 By this fact and using set forcings which separate cardinal characteristics $\mathfrak{d}, \cof(\mcalM)$ and $\cof(\mcalN)$, we can show the following results on the pre-Tukey relation in $\LR$.
 \begin{prop}
  Assume $\AC$ in $V$.
  Suppose that there exist proper class many Woodin cardinals.
  Then $\LR\models (\mcalN, \subseteq) \npreceq_{pT} (\mcalM, \subseteq) \npreceq_{pT} (\omega^\omega, \leq^*)$. 
 \end{prop}

\section{Questions}

In this paper, we introduce the pre-Tukey relation and show in Lemma \ref{Tukey_imply_preTukey} and \ref{preTukey_imply_Tukey} that it coincides with the Tukey relation under $\AC$.
Does the converse hold?

\begin{question}
Does the coincidence of the Tukey relation and the pre-Tukey relation imply $\AC$?
\end{question}

Between relational systems, which generalize ordered sets, there is a relation defined by the Galois-Tukey connection, a generalization of the Tukey reducibility. 
In the context of cardinal characteristics in $\ZFC$, such connections have been thoroughly studied.

\begin{question}
  Can the Galois-Tukey connection be suitably generalized to $\ZF$? 
  If so, what is the behavior of that relation in the models of $\ZF+\DC+(\star)$?
\end{question}

\section*{Acknowledgements}
We thank Yusuke Hayashi for valuable discussions on the definition and properties of the pre-Tukey relation.
We also thank Toshimichi Usuba for comments on the Tukey reducibility in the context without choice.
We appreciate Jonathan Schilhan's meaningful remarks on forcing absoluteness in Section 4.5.

\bibliographystyle{plain}
\bibliography{ref}

\end{document}